\documentclass[review]{elsarticle}
\pdfoutput=1

\usepackage{lineno}
\usepackage{hyperref}
\usepackage{amssymb,amsmath,amsthm,mathtools}
\usepackage{changepage}
\usepackage[english]{babel}
\usepackage[T1]{fontenc}
\usepackage{lmodern}
\usepackage{microtype}
\usepackage[textsize=tiny]{todonotes}
\usepackage{tikz}
\usetikzlibrary{graphs,graphs.standard,calc,math}
\usetikzlibrary{arrows, arrows.meta}
\usetikzlibrary{automata,arrows,positioning,calc}
\usepackage{tikz-3dplot}
\usepackage{subcaption}
\usepackage{graphicx}
\usepackage{caption}
\usepackage{pgfplots}
\usepackage[short]{optidef}
\pgfplotsset{width=7cm,compat=1.18}

\usepackage{enumitem}
\setlist{noitemsep,labelwidth=*,leftmargin=*,align=left}
\setlist[enumerate,1]{label=(\alph*)}
\setlist[description]{font=\normalfont,leftmargin=!}

\usepackage[ruled,noend]{algorithm2e}
\renewcommand{\SetProgSty}[1]{\renewcommand{\ProgSty}[1]{\textnormal{\csname#1\endcsname{##1}}\unskip}}
\SetProgSty{}
\SetFuncSty{textsc}
\SetFuncArgSty{}
\SetArgSty{}
\SetCommentSty{emph}
\DontPrintSemicolon

\SetKwComment{tcp}{\textup{\texttt{\#}} }{}
\SetKwComment{tcc}{''' }{ '''}
\SetKwProg{Def}{def}{:}{}

\newcommand{\algass}{\ensuremath{\leftarrow}}

\renewcommand{\O}[1]{\ensuremath{\mathcal{O}(#1)}}

\SetKw{Continue}{continue}
\SetKw{Break}{break}

\SetKwInput{Input}{Input}
\SetKwInput{Output}{Output}
\SetKw{Return}{\textbf{Return:}}
\SetKw{And}{\textbf{and}}
\SetKwComment{Comment}{/* }{ */}

\usepackage[capitalise]{cleveref}
     
\theoremstyle{definition}
\newtheorem{theorem}{Theorem}[section]

\newtheorem{corollary}[theorem]{Corollary}
\newtheorem{definition}[theorem]{Definition}
\newtheorem{remark}[theorem]{Remark}
\newtheorem{lemma}[theorem]{Lemma}
\newtheorem{problem}[theorem]{Problem}

\newtheorem{observation}[theorem]{Observation}

\newtheorem{assumption}[theorem]{Assumption}
\theoremstyle{definition}
\newtheorem{example}[theorem]{Example}

\newcommand{\M}{\ensuremath{\mathcal{M}}}

\newcommand{\F}{\ensuremath{\mathcal{F}}}
\newcommand{\T}{\ensuremath{\mathcal{T}}}
\newcommand{\I}{\ensuremath{I}}
\newcommand{\R}{\ensuremath{\mathbb{R}}}
\newcommand{\Set}[1]{\left\{ #1 \right\}}

\pgfplotsset{compat=1.18}

\biboptions{authoryear}

\begin{document}
	
\begin{frontmatter}
		
    \title{Multi-parametric matroids - Applications to interdiction and weight set decomposition}
    
    \author[a]{Nils Hausbrandt}\corref{mycorrespondingauthor}
    \cortext[mycorrespondingauthor]{Corresponding author}
    \ead{nils.hausbrandt@math.rptu.de}
    \author[a]{Stefan Ruzika}
    \ead{stefan.ruzika@math.rptu.de}
      
    \address[a]{Department of Mathematics, University of Kaiserslautern-Landau, 67663 Kaiserslautern, Germany}
    
    \begin{abstract}
        In this article, we investigate the multi-parametric matroid problem.
        The weights of the elements of the matroid's ground set depend linearly on an arbitrary but fixed number of parameters, each of which is taken from a real interval.
        The goal is to compute a minimum weight basis for each  possible combination of the parameters.
        For this problem, we propose an algorithm that requires a polynomial number of independence tests and discuss two useful applications.
        First, the algorithm can be applied to solve a multi-parametric version of a special matroid interdiction problem, and second, it can be utilized to compute the weight set decomposition of the multi-objective (graphic) matroid problem.
        For the latter, we asymptotically improve the current state-of-the-art algorithm by a factor that is almost proportional to the number of edges of the graphic matroid.
    \end{abstract}
    
    
\end{frontmatter}

\section{Introduction}
Matroid theory was initially established by \cite{whitney1935on} and has long been a well-established field of modern mathematics, cf.\ \cite{wilson1973introduction, welsh2010matroid, oxley2011matroid}.
Matroids generalize fundamental concepts such as linear independence in linear algebra or important basics of graph theory and provide numerous applications, for instance in combinatorics, coding theory or network theory, cf.\ \cite{el2010index, tamo2016optimal, ouyang2021importance}.

Parametric matroids are characterized by the property that each weight of the elements of the matroid's ground set is given as a function depending on some real-valued parameters from a given parameter interval.
The goal is to compute a basis of minimum weight for each possible combination of the parameters.
A solution to the problem is given by a partition of the parameter interval into subsets such that each subset is assigned a unique basis that is optimal for all parameter values within the subset.
The dependence of the minimum weight basis on the parameters naturally implies two closely related research questions.
In addition to the algorithmic computation of a solution set, the question of its size is also of crucial importance.
The problem only admits a polynomial time algorithm if the number of subsets (and corresponding optimal bases) needed to decompose the parameter interval is polynomial in the size of its input.

In the literature, attention has so far mainly been paid to the special case of one-parametric graphic matroids, where the weight of each edge of the graph depends on exactly one parameter, cf.\ \cite{fernandez1996using, agarwal1998parametric}.
\cite{gusfield1979bound} and \cite{dey1998improved} proved upper bounds on the cardinality of the solution set and, recently, \cite{eppstein2023stronger} improved his lower bound, cf.\ \cite{eppstein1995geometric}.
For arbitrary one-parametric matroids, tight bounds where established by \cite{eppstein1995geometric, dey1998improved, katoh2002parametric}.
There exists several general methods that do not explicitly use the special matroid structure of the problem, cf.\ \cite{eisner1976mathematical,bazgan2022approximation, helfrich2022approximation}.
We refer the reader to \cite{nemesch2025survey} for a survey on general parametric optimization, in which the close relationship between multi-parametric and multi-objective optimization is emphasized.
This relationship implies that the multi-parametric minimum spanning tree problem admits a solution set of polynomial cardinality, cf.\ \cite{seipp2013adjacency}.
The algorithm of \cite{seipp2013adjacency} for the weight set decomposition in the multi-objective minimum spanning tree problem proposes a polynomial time algorithm for the multi-parametric minimum spanning tree problem.
To our knowledge, however, there is no work that addresses the multi-parametric version of arbitrary matroids.

This article intends to close this gap by investigating the multi-parametric matroid problem.
The weight of each element of the matroid's ground set is assumed to depend linearly on an arbitrary but fixed number of real-valued parameters.
The goal is to compute a minimum weight basis for each possible combination of the parameters.
We show that the cardinality of the solution set of the multi-parametric matroid problem is polynomial in the input length.
We develop an algorithm that solves the problem with a polynomial number of independence test.
Our algorithm has two useful applications.

First, it can be applied to solve the multi-parametric version of the matroid one-interdiction problem (\cite{hausbrandt2024parametric}) using only a polynomial number of independence test.
Here, the goal is to compute, for each possible parameter value, an optimal interdiction strategy, i.~e.\ a so-called most vital element.
This is an element that maximally increases the weight of a minimum weight basis when removed from the matroid.
Besides developing an algorithm, we show that a solution to the problem only requires a polynomial number of interdiction strategies.

Second, our algorithm can be applied to the multi-objective version of the (graphic) matroid problem.
We obtain a polynomial time algorithm for the computation of all extreme supported spanning trees and the corresponding weight set decomposition for an arbitrary but fixed number of objectives.
In contrast to other general weight set decomposition algorithms, cf.\ \cite{benson2000outcome,przybylski2010recursive,halffmann2020inner}, the running time of our algorithm is easy to analyze and can be stated explicitly. 
We use adjacency of optimal bases in combination with a breadth-first search technique on the weight set components to asymptotically improve upon the algorithm of \cite{seipp2013adjacency}.
Consequently, our algorithm is asymptotically faster than all existing methods that allow an easy and explicit computation of their asymptotic running time, cf.\ \cite{seipp2013adjacency, correia2021finding}.

The remainder of this article is organized as follows.
In \cref{sec:prelims_p_1}, we introduce the notation, define the multi-parametric matroid problem, and recall some required concepts from computational geometry.
In \cref{sec:para_matroid_algos_p_1}, we adapt the algorithm of \cite{seipp2013adjacency} in a parametric context to arbitrary matroids and present our improved procedure.
In \cref{sec:wsd_p_1}, we apply our algorithm to compute the weight set decomposition in the (graphic) matroid problem.
\cref{sec:interdiciton_p_1} explains the application of our algorithm to the multi-parametric matroid one-interdiction problem.

\section{Preliminaries}\label{sec:prelims_p_1}
In this section, we explain the notation and introduce the required definitions for the multi-parametric matroid problem.
In addition, we recall the necessary concepts from computational geometry.
For a singleton $\Set{a}$ and a set $A$, we use the notation $A+a$ or $A-a$ instead of $A\cup\Set{a}$ or $A\setminus\Set{a}$.

\textbf{Matroids.}
Let $E$ be a finite set and $\emptyset\neq\F\subseteq 2^E$ be a non-empty subset of the power set of $E$.
The tuple $\M = (E,\F)$ is called \emph{matroid} if the following properties hold:
\begin{enumerate}
    \item The empty set $\emptyset$ is contained in $\F$.
    \item If $A\in\F$ and $B\subseteq A$, then also $B\in\F$.
    \item If $A,B \in\F$ and $\vert B\vert < \vert A\vert$, then there exists an element $a\in A\setminus B$ such that $B + a\in\F$.
\end{enumerate}
A subset $F\in\F$ of $E$ is called \emph{independent}, while all other subsets of $E$ are called \emph{dependent}.
A \emph{basis} is an inclusion-wise maximal independent set.
The \emph{rank} $rk(\M)$ of the matroid is the cardinality of a basis.
We define $\vert E\vert\coloneqq m$ and $rk(\M)\coloneqq k$.
We denote the time required to perform a single independence test, i.~e.\ a test to determine whether a subset of $E$ is independent, by $f(m)$.

\textbf{Multi-parametric matroids.}
In the multi-parametric matroid problem, each element $e\in E$ of the matroid's ground set is associated with a weight function depending linearly on $p$ parameters, where $p$ is a fixed natural number.
More precisely, the \emph{$p$-parametric weight} of an element $e\in E$ is given by $$w(e,\lambda)\coloneqq w(e,\lambda_1,\dotsc,\lambda_p)= a_e + \lambda_1 b_{1,e} +\dotsc+\lambda_p b_{p,e},$$ where $a_e, b_{1,e},\dotsc, b_{p,e}\in \mathbb{Q}$ are rational numbers.
For each $i=1,\dotsc,p$, the parameter $\lambda_i$ is taken from a real interval $\I_i\subseteq\mathbb{R}$.
The Cartesian product $\I\coloneqq\I_1\times\dotsc\times\I_p$ of all intervals is called the \emph{(p-dimensional) parameter interval}.
The weight of a basis $B$ naturally depends on the parameters and is defined accordingly as $w(B,\lambda)\coloneqq w(B,\lambda_1,\dotsc,\lambda_p)\coloneqq\sum_{e\in B} w(e,\lambda)$.
For a fixed $\lambda\in\I$, a minimum weight basis $B_\lambda^*$ at $\lambda$ is a basis $B^*$ of $\M$ with $w(B^*,\lambda)\leq w(B,\lambda)$ for all bases $B$ of $\M$. 
The $p$-parametric matroid problem can be stated as follows.
\begin{problem}[$p$-Parametric matroid problem]\label{prob:matroid_p}
    Given a matroid~$\M$ with $p$-parametric weights $w(e,\lambda)$ and the parameter interval~$\I\subseteq \mathbb{R}^p$, compute, for each possible parameter vector $\lambda\in\I$, a minimum weight basis $B_\lambda^*$.
\end{problem}
The corresponding objective function value is specified by the optimal value function, which maps the parameter vector to the weight of a minimum weight basis.
\begin{definition}[Optimal value function]
    The function $w\colon I \to \mathbb{R}, \lambda \mapsto w(B_\lambda^*,\lambda)$ mapping the parameter vector to the weight of a minimum weight basis is called \emph{optimal value function} (of the $p$-parametric matroid problem). 
\end{definition}
To get a better understanding of the structure of the $p$-parametric matroid problem, it is crucial to determine the sorting of all weights $w(e,\lambda)$ over the entire parameter interval $\I$.
To do so, we consider the intersection of each pair of weight functions $w(e,\lambda)$ of the elements $e\in E$.
This results in a set of separating hyperplanes that divide the $\R^p$ into open half-spaces.
\begin{definition}[Separating hyperplane (\cite{seipp2013adjacency})]
	For two elements $e,f\in E$ with weight functions $w(e,\lambda)$ and $w(f,\lambda)$, the \emph{separating hyperplane} $h(e,f)\subseteq\R^p$ is defined as $h(e,f)~=~\Set{\lambda\in\mathbb{R}^p\colon w(e,\lambda)=w(f,\lambda)}$. It divides $\mathbb{R}^p$ into two open half-spaces
    \begin{enumerate}
        \item[] $h^<(e,f)=\Set{\lambda\in\mathbb{R}^p\colon w(e,\lambda)<w(f,\lambda)}$ and
        \item[] $h^>(e,f)=\Set{\lambda\in\mathbb{R}^p\colon w(e,\lambda)>w(f,\lambda)}$,
    \end{enumerate}
	where the former consists of all parameter vectors for which $e$ attains a smaller weight than $f$ and, vice versa, the latter contains all parameter vectors for which $f$ has smaller weight than $e$.
    The set of all separating hyperplanes is denoted by $H^=\coloneqq\Set{h(e,f)\colon e,f\in E}$.
\end{definition}
We can now state our assumptions which are essentially without loss of generality.
\begin{assumption}\label{ass:p_1}
	Throughout this paper, we make the following assumptions.
    \begin{enumerate}
        \item No two separating hyperplanes in $H^=$ are identical.
        \item No hyperplane $h(e,f)\in H^=$ is vertical, meaning that it contains a line parallel to the $p$-th coordinate axis.
    \end{enumerate}
\end{assumption}
The first assumption ensures that at the transition from one half-space $h^<(e,f)$ ($h^>(e,f)$) via the hyperplane $h(e,f)$ into the other half-space $h^>(e,f)$ ($h^<(e,f)$) it is known how the sorting of the weights of the elements in $E$ changes.
This property turns out to be useful later in our algorithm.
The second assumption is a technical requirement for an algorithm that we need later to compute the polyhedra obtained by intersecting the hyperplanes $h(e,f)\in H^=$, cf.\ \cref{theo:compute_arrangement_p_1}.
If a hyperplane $h(e,f)$ is vertical or two hyperplanes $h(e,f)$ and $h(e,g)$ or $h(e,h)$ are identical, this can be fixed by an infinitesimally small change of the weight $w(e,\lambda)$.
Such a perturbation has only an infinitesimally small effect on the sorting of the weights and, thus, on the optimal value function $w$.
Further, for a given instance, it is easy to check whether one of the assumptions is not fulfilled.
Note, that the first assumption also implies that no two weight functions $w(e,\lambda)$ and $w(f,\lambda)$ are identical.

As in the one-parametric matroid problem, the optimal value function maintains piecewise linearity, continuity, and concavity.
\begin{lemma}\label{lem:opt_value_fct_concave_p_1}
    The optimal value function $w$ is piecewise linear, continuous and concave.
\end{lemma}
\begin{proof}
    Each of the at most $\binom{m}{k}$ many bases $B$ of $\M$ yields a linear function $w_B\colon I \to \mathbb{R},\lambda\mapsto w(B,\lambda)$.
    The optimal value function $w$ is given by the lower envelope, i.~e.\ the point-wise minimum, $w(\lambda)=\min\Set{w_B(\lambda)\colon B \; \text{basis of} \; \M}$ of all these functions $w_B$.
    It follows that $w$ is piecewise-linear, continuous, and concave.
\end{proof}
Piecewise linearity and continuity of the optimal value function $w$ ensure that its graph in $\R^{p+1}$ consists of a collection of $p$-dimensional polyhedra, which are connected by faces of smaller dimension.
Each $p$-dimensional polyhedron belongs to a unique minimum weight basis $B^*$, for which the function $w_{B^*}$ forms the lower envelope $w$ for all parameter vectors from the corresponding polyhedron.
To put it differently, if a basis is optimal for one point in the relative interior of the polyhedron, it is optimal for all points of the polyhedron.
The concavity causes each basis $B$ of $\M$ to contribute at most one $p$-dimensional polyhedron to the graph, resulting in an upper bound of $\O{m^k}$.
Our goal is to derive a tighter bound on the number of these polyhedra in the worst case, as this naturally determines the complexity of the problem.
For this reason, we often consider the most general case $\I=\R^p$.
A smaller interval may only yield fewer polyhedra in the graph of $w$ and, thus, to a shorter running time of our algorithms.
In \cref{sec:para_matroid_algos_p_1}, we investigate the structure of \cref{prob:matroid_p} in more detail, resulting in an upper bound of polynomial size on the number of $p$-dimensional polyhedra in the graph of $w$.

\cref{ex:opt_value_fct_p_1} shows an illustration of an optimal value function of an instance of the $p$-parametric graphic matroid problem, which we now define.

\textbf{Multi-parametric minimum spanning trees.}
All results presented in this article hold for arbitrary (parametric) matroids.
To illustrate our concepts, we use the special case of graphic matroids, which we now introduce.
Let $G=(V,E)$ be an undirected and connected graph with node set $V$ and edge set $E$.
We set $\vert V\vert\coloneqq n$ and $\vert E\vert\coloneqq m$.
A \emph{spanning tree} $T$ of $G$ is a tree connecting all nodes of $G$.
We denote the set of all spanning trees of $G$ by $\T$.
Each edge $e\in E$ is associated with a $p$-parametric weight $w(e,\lambda)\coloneqq w(e,\lambda_1,\dotsc,\lambda_p)= a_e + \lambda_1 b_{1,e}+\dotsc+\lambda_p b_{p,e},$ where $a_e, b_{1,e},\dotsc, b_{p,e}\in \mathbb{Q}$ and $\lambda\in\I$.
A minimum weight basis at $\lambda\in\I$ corresponds to a minimum spanning tree $T_\lambda^*$ at $\lambda$.
In analogy to \cref{prob:matroid_p}, we obtain the following formulation.
\begin{problem}[$p$-Parametric minimum spanning tree problem]\label{prob:mst_p}
    Given a graph $G=(V,E)$ with $p$-parametric edge weights $w(e,\lambda)$ and the parameter interval~$\I~\subseteq~\mathbb{R}^p$, compute, for each possible parameter vector $\lambda\in\I$, a minimum spanning tree $T_\lambda^*\in\T$.
\end{problem}
Accordingly, the optimal value function $w$ maps the parameter vector to the weight of a minimum spanning tree.
We finish this subsection with an example of a 2-parametric minimum spanning tree problem and illustrate its optimal value function.
\begin{example}\label{ex:opt_value_fct_p_1}
    Consider the graph from \cref{fig:ex_graph_parametric_p_1} with four edges $e,f,g,$ and $h$, each having a 2-parametric weight.
    The parameter interval is given by $\I\coloneqq\R^2$.
    \cref{fig:opt_value_fct_p_1} shows a fragment of the corresponding piecewise linear, continuous and concave optimal value function $w$, which consists of four 2-dimensional polyhedra marked by different levels of gray.
    The spanning tree $T_1=\Set{e,h}$ is optimal for (all points of) the polyhedron bottom left, $T_2=\Set{h,g}$ is optimal for the polyhedron bottom right, $T_3=\Set{f,g}$ for the polyhedron top right and $T_4=\Set{e,f}$ for the polyhedron top left.  
    We show in \cref{ex:neighborhood_bases_p_1} how these polyhedra and the corresponding minimum spanning trees can be efficiently computed.
\end{example}

\begin{figure}
    \centering
    \begin{tikzpicture}[scale=0.7] 
        % Nodes
        \node (A) at (-1,0) [circle,draw,inner sep=2.2pt] {};
        \node (B) at (6,0) [circle,draw,inner sep=2.2pt] {};
        \node (C) at (2.5,4) [circle,draw,inner sep=2.2pt] {};
        % Edges
        \path[thick] (A) edge [bend left=25] node[above] {\footnotesize$w(e,\lambda)=6\lambda_1+4\lambda_2$} (B);
        \path[thick] (A) edge [bend right=25] node[below] {\footnotesize$w(g,\lambda)=1+2\lambda_1+8\lambda_2$} (B);
        \path[thick] (C) edge[bend right=25] node[left] {\footnotesize$w(f,\lambda)=2+4\lambda_1+2\lambda_2$} (A);
        \path[thick] (C) edge[bend left=25] node[right] {\footnotesize$w(h,\lambda)=6+4\lambda_1+12\lambda_2$} (B);
    \end{tikzpicture}
    \caption{An instance of the 2-parametric minimum spanning tree problem.}
    \label{fig:ex_graph_parametric_p_1}

\vspace{20pt}

    \begin{tikzpicture}[scale=1.4]
        \begin{axis}[xmin=-6.5, xmax=6.5,axis line style ={line width=.5pt,draw=gray!40},enlargelimits={abs=0.5},grid=both,major grid style={line width=.2pt,draw=gray!10}
        ] %, ymin=-4, ymax=4
    
        \addplot3 [surf, mesh/rows=2, fill=black!3, fill opacity=1, line width=0.1] coordinates {
            (-3,-3.25,-76) (3,-3.25,-64)
            (-0.15,-0.4,-1.9) (4,-0.4,23)
        };
        \addplot3 [surf, mesh/rows=2, fill=black!18, fill opacity=1, line width=0.1] coordinates {
            (-5,-3,-76) (-4,-0.4,-40.4) 
            (-3,-3.25,-76) (-0.15,-0.4,-1.9)
        };
        \addplot3 [surf, mesh/rows=2, fill=black!29, fill opacity=1, line width=0.1] coordinates {
            (-4,-0.4,-40.4) (-0.15,-0.4,-1.9)
            (-3,2.75,-11.5) (3,2.75,48.5)
        };
        \addplot3 [surf, mesh/rows=2, fill=black!12, fill opacity=1, line width=0.1] coordinates {
            (-0.15,-0.4,-1.9) (4,-0.4,23)
            (3,2.75,48.5) (6,2,59)
        };
        \draw[line width=0.5] (-0.15,-0.4,-1.9) -- (4,-0.4,23);
        \draw[line width=1,black!3] (3,-3.25,-64) -- (4,-0.4,23);
        \draw[line width=0.7,dotted] (3,-3.25,-64) -- (4,-0.4,23);
        \draw[line width=1,black!3] (-3,-3.25,-76) -- (3,-3.25,-64);
        \draw[line width=0.7,dotted] (-3,-3.25,-76) -- (3,-3.25,-64);
        \draw[line width=0.5] (-0.15,-0.4,-1.9) -- (-3,-3.25,-76);
    
        \draw[line width=1,black!18] (-5,-3,-76) -- (-4,-0.4,-40.4);
        \draw[line width=0.7,dotted] (-5,-3,-76) -- (-4,-0.4,-40.4);
        \draw[line width=0.5] (-0.15,-0.4,-1.9) -- (-4,-0.4,-40.4);
        %\draw[line width=0.5] (-0.15,-0.4,-1.9) -- (-3,-3.25,-76);
        \draw[line width=1,black!18] (-5,-3,-76) -- (-3,-3.25,-76);
        \draw[line width=0.7,dotted] (-5,-3,-76) -- (-3,-3.25,-76);
    
        %\draw[line width=0.5] (-0.15,-0.4,-1.9) -- (-4,-0.4,-40.4);
        \draw[line width=1,black!29] (-4,-0.4,-40.4) -- (-3,2.75,-11.5);
        \draw[line width=0.7,dotted] (-4,-0.4,-40.4) -- (-3,2.75,-11.5);
        \draw[line width=0.5] (-0.15,-0.4,-1.9) -- (3,2.75,48.5);
        \draw[line width=1,black!29] (-3,2.75,-11.5) -- (3,2.75,48.5);
        \draw[line width=0.7,dotted] (-3,2.75,-11.5) -- (3,2.75,48.5);
    
        \draw[line width=1,black!12] (6,2,59) -- (3,2.75,48.5);
        \draw[line width=0.7,dotted] (6,2,59) -- (3,2.75,48.5);
        \draw[line width=1,black!12] (6,2,59) -- (4,-0.4,23);
        \draw[line width=0.7,dotted] (6,2,59) -- (4,-0.4,23);
        %\draw[line width=0.5] (-0.15,-0.4,-1.9) -- (4,-0.4,23);
        %\draw[line width=0.5] (-0.15,-0.4,-1.9) -- (3,2.75,48.5);
        
        \end{axis}
        
        \node (m) at (1.725,1.94) {}; % origin
    
        \node (mr) at (1.82,2.01) {};
        \node (l) at (-0.25,0.419) {};
        \draw[->,line width=0.5] (mr) -- (l);
        
        \node (o) at (1.725,4.8) {};
        \node (mu) at (1.725,1.85) {};
        \draw[->,line width=0.5] (mu) -- (o);
    
        \node (ml) at (1.813,1.876) {};
        \node (r1) at (1.835,1.86) {};
        \draw[line width=0.5] (ml) -- (r1);
    
        \node (r2) at (3.21,1.69) {};
        \node (r3) at (5.75,1.26) {};
        \draw[->,line width=0.5] (r2) -- (r3);
    
        \node (r4) at (5.9,1.25) {$\lambda_1$};
        \node (l1) at (-0.36,0.4) {$\lambda_2$};
        \node (o1) at (1.725,4.94) {$w(\lambda)$};
        
    \end{tikzpicture}
    \caption{A fraction of the piecewise linear, continuous and concave optimal value function of the instance of the 2-parametric minimum spanning tree problem from \cref{fig:ex_graph_parametric_p_1}.}
    \label{fig:opt_value_fct_p_1}
\end{figure}

\textbf{Arrangements of hyperplanes.}
As mentioned before, it is essential to know how the sorting of all weights behaves over the entire parameter interval $\I$.
Therefore, it is crucial to intersect the separating hyperplanes $h(e,f)\in H^=$.
To obtain this intersection and the regions between the separating hyperplanes, we use the concept of arrangement of hyperplanes, see \cite{edelsbrunner1987algorithms,orlik2013arrangements}.
\begin{definition}[\cite{edelsbrunner1986constructing}]
	Let $H=\Set{h_1,\dotsc,h_q}$ denote a set of $q$ hyperplanes in $\mathbb{R}^d$, for $d\geq2$.
	Let $h^<_i$ and $h^>_i$ denote the open half-spaces above and below $h_i$, for $i=1,\dotsc,q$.
	The \emph{arrangement} $A(H)$ consists of faces $f$ with $$f=\bigcap\limits_{i=1,\dotsc,q} f(h_i),$$ where $f(h_i)\in\Set{h_i,h^<_i,h^>_i}$.
    A face $f$ is called a \emph{$t$-face}, for $0\leq t\leq d$, if its dimension is $t$.
    A $d$-face in $\R^d$, i.~e.\ a full-dimensional face, is called a \emph{cell}.
\end{definition}
The cells of the arrangement of the hyperplanes resulting from the intersection of each pair of weights $w(e,\lambda)$ and $w(f,\lambda)$ are of crucial importance. 
We see in the next section that each cell can be assigned a unique basis which is optimal for each parameter value in the relative interior of the cell.
It is therefore essential to have a bound on the number of cells of an arrangement.
\begin{theorem}[\cite{seipp2013adjacency}]\label{theo:number_of_cells_p_1}
    An arrangement of $q$ hyperplanes in $\mathbb{R}^d$ contains a maximum of $\sum_{i=0}^{d} \binom{q}{i} \leq 2q^d$ many cells.
\end{theorem}
\cite{edelsbrunner1986constructing} developed an algorithm to compute all faces of an arrangement of hyperplanes $A(H)$.
The iterative algorithm adds one new hyperplane to an existing arrangement per iteration and updates it until all hyperplanes are included.
The algorithm is optimal in the sense that the time to construct the arrangement does not asymptotically exceed the time to store it. 
\begin{theorem}[\cite{edelsbrunner1986constructing}]\label{theo:compute_arrangement_p_1}
    Let $H$ be a set of $q$ hyperplanes in $\mathbb{R}^d$, for $d\geq2$.
    Then, the arrangement $A(H)$ can be computed in $\O{q^d}$ time, and this is optimal.
\end{theorem}
In the next section, we apply the algorithm to compute the arrangement of the separating hyperplanes from the set $H^=$.
Note, that \cref{ass:p_1} guarantees that these hyperplanes are non-vertical.

In the algorithm of \cite{edelsbrunner1986constructing}, each face is represented by one point from its relative interior.
The arrangement $A(H)$ is then stored using a graph $G'(A(H))$ called \emph{incidence lattice}.
Each face $f\in A(H)$ is represented by a separate node in $G'(A(H))$ and incidences between faces are displayed by edges.
If one is only interested in the cells of the arrangement, the smaller incidence graph proves helpful.
\begin{definition}[Incidence graph (\cite{grunbaum1967convex})]
    For an arrangement $A(H)$ of hyperplanes in $\R^d$, the \emph{incidence graph} $G(A(H))$ is defined as follows.
    For each cell $c\in A(H)$, there exists a node $v_c$ in $G(A(H))$.
    Two nodes $v_{c_1}$ and $v_{c_2}$ are connected by an edge if and only if the two cells $c_1$ and $c_2$ share a common face of dimension $d-1$.
\end{definition}
In the next section, we traverse the incidence graph to compute a minimum weight basis for each node and, thus, each cell of the arrangement.
In order to analyze the running time of this algorithm, we need the following theorem to bound the number of edges in the incidence graph.
\begin{theorem}[\cite{grunbaum1967convex}]\label{theo:number_of_faces_p_1}
    For $0\leq t<d$, the number of $t$-faces of an arrangement of $q$ hyperplanes in $\R^d$ is bounded from above by
    $$\binom{q}{d-t}\sum_{i=0}^{t}\binom{q-d-1+t}{i}.$$
\end{theorem}

\section{Algorithms}\label{sec:para_matroid_algos_p_1}
In this section, we describe two algorithms that solve the multi-parametric matroid problem with a polynomial number of independence tests.
The first algorithm is a generalization of the algorithm of \cite{seipp2013adjacency} for the multi-objective minimum spanning tree problem to arbitrary multi-parametric matroids.
The second algorithm extends this procedure by a special neighborhood relation between minimum weight bases.
In addition, we use a breadth-first search technique to efficiently traverse the cells of the arrangement of the separating hyperplanes defining the sorting of the weight functions.
As a consequence, we improve upon the asymptotic running time of the first algorithm by a factor of at least $m$.

\textbf{Algorithm 1.}
The matroid structure implies that only the sorting of the weights of the elements, but not their numerical value, is relevant for the computation of a minimum weight basis, cf.\ \cite{oxley2011matroid}.
In the $p$-parametric matroid problem, the weights of the elements describe a $p$-dimensional hyperplane in $\R^{p+1}$.
For each pair of weight functions $w(e,\lambda)$ and $w(f,\lambda)$, we consider the separating hyperplane $h(e,f)\in H^=$.
By \cref{ass:p_1}, each such hyperplane has a dimension of $p-1$.
The hyperplane $h(e,f)$ determines where the element $e$ has a smaller weight than $f$.
The intersection of all hyperplanes in $H^=$ thus defines the sorting of the weights over $\I$ and forms an arrangement of hyperplanes $A(H^=)$ in $\R^p$.
There is a one-to-one correspondence between the resulting faces of the arrangement and the sorting of the weights.
If a basis is optimal at a point from the relative interior of a face $f$, it is optimal for all $\lambda\in f$.
As the cells of $A(H^=)$ decompose the parameter interval $I$ and, since we are only interested in one solution for each parameter vector $\lambda\in\I$, we obtain the following observation.
\begin{observation}\label{obs:cells_suffice_p_1}
    It suffices to compute a minimum weight basis at one point from the relative interior of each cell of the arrangement $A(H^=)$ to solve the $p$-parametric matroid problem.
\end{observation}
The algorithm of \cite{edelsbrunner1986constructing} does not only compute the faces of an arrangement but also one point from the relative interior of each face.
If the algorithm is applied to the arrangement $A(H^=)$, we only have to run the greedy algorithm for the point of the relative interior of each cell.
We now state \cref{alg:p-parametric_matroid_p_1}, which is a direct adaptation of the algorithm of \cite{seipp2013adjacency} in the parametric context to arbitrary matroids.

\begin{algorithm}[t]
    \LinesNumbered
    \Input{A matroid \M{} with $p$-parametric weights $w(e,\lambda)$ and parameter interval~$I \subseteq \mathbb{R}^p$.}
    \Output{A solution of the $p$-parametric matroid problem.}
    
    Compute the set $H^==\Set{h(e,f)\colon e,f\in E}$\;
    Compute $A(H^=)$ with the algorithm of \cite{edelsbrunner1986constructing}\;
    %Let $P$ be the set that contains one inner point for each face \;
    \For{each cell $c\in A(H^=)$ with corresponding inner point $\lambda_c$}{
        Compute a minimum weight basis $B_{\lambda_c}^*$ at $\lambda_c$\;
    }
    \Return $A(H^=)$ and the minimum weight basis of each cell $c\in A(H^=)$\;
    \caption{An algorithm for the $p$-parametric matroid problem.}
    \label{alg:p-parametric_matroid_p_1}
\end{algorithm}

\begin{theorem}
    Let $p\geq2$.
    \cref{alg:p-parametric_matroid_p_1} solves the $p$-parametric matroid problem in $\O{m^{2p+1}(f(m)+\log m)}$ time.
\end{theorem}
\begin{proof}
    By \cref{obs:cells_suffice_p_1} it suffices to consider the cells of the arrangement $A(H^=)$.
    The correctness follows as an optimal basis at a point $\lambda$ from the relative interior of a cell $c\in A(H^=)$ is optimal for all $\lambda\in c$.
    The set $H^=$ contains at most $\binom{m}{2}$ many hyperplanes, each of which can be computed in constant time.
    The algorithm of \cite{edelsbrunner1986constructing} computes the arrangement $A(H^=)$ and one point from the relative interior of each cell $c\in A(H^=)$ in $\O{q^p}$ time, where $q\in \O{m^2}$.
    Finally, we execute one call of the greedy algorithm in $\O{m\log m+mf(m)}$ time for each of the at most $2q^p$ many cells of $A(H^=)$, cf.\ \cref{theo:number_of_cells_p_1}, where $q\in \O{m^2}$.
    In summary, we obtain a total running time of
    $$\O{m^2+m^{2p}+2m^{2p} m (f(m)+\log m)} = \O{m^{2p+1}(f(m)+\log m)}.$$
\end{proof}
\begin{remark}
    In order to obtain the maximum possible number of cells in the arrangement $A(H^=)$, at most 2 of the hyperplanes intersect in the same face of dimension $p-2$.
    However, the fact that two weights are equal on each hyperplane reduces the maximum number of cells in the arrangement.
    For each triple of the form $(h(e,f),h(e,g),h(f,g))$, there exists the face $h(e,f,g)=\Set{\lambda\in\R^p\colon w(e,\lambda)=w(f,\lambda)=w(g,\lambda)}$ of dimension $p-2$ in which these three hyperplanes intersect.
    For $p=2$, an example of an intersection point $h(e,f,g)$ of such a triple can be found in \cref{fig:neighborhood_p_1} at point (0.625,0.375).
    In general, each of these $\binom{m}{3}$ many triples thus reduces the number of cells by at least 1.
    Nevertheless, this reduction has no influence on the asymptotic bound on the maximum number of cells.
\end{remark}
We are now in the position to derive a polynomial bound on the number of $p$-dimensional polyhedra in the graph of the optimal value function $w$.
Recall, that the minimum weight basis is identical for all points in the relative interior of such a polyhedron.
It follows by construction, that the $p$-dimensional polyhedra are defined by the separating hyperplanes in $H^=$.
More precisely, a $(p-1)$-dimensional hyperplane $h$ can only be a facet of such a polyhedron, if $h$ is contained in one of the separating hyperplanes from the set $H^=$.
The $p$-dimensional polyhedra correspond to the cells of the arrangement $A(H^=)$, however, one polyhedron can subsume
different neighboring cells.
Consequently, the number of different $p$-dimensional polyhedra in the graph of $w$ is bounded by the number of cells of $A(H^=)$ and we obtain the following corollary.
\begin{corollary}\label{cor:opt_value_fct_p_1}
    The graph of the optimal value function $w$ contains at most $\O{m^{2p}}$ many $p$-dimensional polyhedra.
\end{corollary}

\textbf{Algorithm 2.}
In the 1-parametric matroid problem, the weights of the elements depend linearly on a real parameter from a one-dimensional parameter interval.
Sweeping through the parameter interval from left to right, one observes that there are only polynomially many points at which the minimum weight basis changes, cf.\ \cite{fernandez1996using, hausbrandt2024parametric}.
In particular, there is a neighborhood relation between the optimal bases before and after such a point.
They differ only by one swap, i.~e.\ their symmetric difference is two.

In the following, we extend this neighborhood relation to the $p$-parametric case and, thus, obtain an asymptotically faster method than \cref{alg:p-parametric_matroid_p_1}.
Since it is sufficient to focus on neighboring cells, we consider the smaller incidence graph $G(A(H^=))$ instead of the incidence lattice $G'(A(H^=))$ from the algorithm of \cite{edelsbrunner1986constructing}.

The principle idea is to first run the greedy algorithm once for a point from the relative interior of an arbitrary cell of the arrangement $A(H^=)$ to compute a minimum weight basis there.
Then, we iteratively pivot through neighboring cells until all cells have been found.
From pivoting from one cell to another, we utilize the fact that, in neighboring cells, only the sorting of two weights $w(e,\lambda)$ and $w(f,\lambda)$ changes.
Thus, a pivot corresponds to updating the minimum weight basis which can be realized by an independence test.
To ensure that all cells are found, we exploit connectivity of $G(A(H^=))$.
\begin{lemma}\label{lem:connectivity_GAH_p_1}
    The incidence graph $G(A(H))$ of an arrangement $A(H)$ of hyperplanes is connected.
\end{lemma}
\begin{proof}
    We show the claim by induction on the number $i$ of hyperplanes that form the arrangement $A(H)$.
    If $i=1$, the claim is trivial.
    Let $i\geq1$ and $v_{c_1}$ and $v_{c_r}$ be two nodes in $G(A(H))$ connected by a path $P=(v_{c_1},\dotsc,v_{c_r})$ with $r\geq2$.
    We add a new hyperplane $h$ to $A(H)$.
    Let $A_i$ and $A_{i+1}$ be the corresponding arrangement before and after $h$ is added, respectively.
    For each cell $c_j$ in $A_i$, we denote the corresponding cell in $A_{i+1}$ by $\bar{c}_j$.
    If $c_j\neq\bar{c}_j$, we denote the two connected cells that bisect $\bar{c}_j$ by $c_j'$ and $c_j''$, respectively.
    
    We show that there exists a $(v_{\bar{c}_1},v_{\bar{c}_2})$-path $P'$ in $G(A_{i+1})$.
    First, let $c_1=\bar{c}_1$.
    If $c_2=\bar{c}_2$, we can choose $P'=(v_{c_1},v_{c_2})$.
    If $c_2\neq\bar{c}_2$, the connectivity of $c_2'$ and $c_2''$ ensures that we can choose $P'\in\Set{(v_{c_1},v_{c_2'}), (v_{c_1},v_{c_2''})}$.

    Second, let $c_1\neq\bar{c}_1$.
    As $c_1$ and $c_2$ are connected in $G(A_i)$, it follows that $c_1'$ or $c_1''$ is connected to $\bar{c}_2$ in $G(A_{i+1})$.
    In the first case, we can choose $P'~\in~\Set{(v_{c_1'},v_{c_2}), (v_{c_1'},v_{c_2'}), (v_{c_1'},v_{c_2''})}$ depending on whether $c_2$ is divided by $h$ or not.
    In the second case, we can replace $c_1'$ by $c_1''$ and obtain the sought path.
    Inductively, there exists a $(v_{\bar{c}_1},v_{\bar{c}_r})$-path $P'$ in $G(A_{i+1})$ and the claim follows.
\end{proof}
The graph connectivity ensures that, starting from any node $v_{c'}$, all nodes $v_c$ in $G(A(H^=))$ can be reached using a breadth-first search.
The choice of the incidence graph $G(A(H^=))$ prevents us from reaching nodes that do not belong to a cell of $A(H^=)$.
The incidence graph $G(A(H^=))$ can be easily derived from the incidence lattice $G'(A(H^=))$ constructed in the algorithm of \cite{edelsbrunner1986constructing} to compute the arrangement $A(H^=)$.
If two cells $c_1$ and $c_2$ share a common face $f$ of dimension $p-1$, there exists a node $v'_f$ in $G'(A(H^=))$ and an edge from $v'_f$ to each of the nodes $v'_{c_1}$ and $v'_{c_2}$ belonging to the cells $c_1$ and $c_2$.
In $G(A(H^=))$, there is an edge between $v_{c_1}$ and $v_{c_2}$ and we store the information to which separating hyperplane $h(e,f)$ this edge belongs.
This requires constant time per face of dimension $p-1$ and, thus, the graph $G(A(H^=))$ can be obtained in time linear in the number of faces of dimension $p-1$ of $A(H^=)$ from $G'(A(H^=))$.
Of course, we store the representative point from the relative interior of each cell.
We are now ready to state our improved algorithm.

\begin{algorithm}[t]
    \LinesNumbered
    \Input{A matroid \M{} with $p$-parametric weights $w(e,\lambda)$ and parameter interval~$I \subseteq \mathbb{R}^p$.}
    \Output{A solution of the $p$-parametric matroid problem.}
    
    Compute the set $H^==\Set{h(e,f)\colon e,f\in E}$\;
    Compute $A(H^=)$ with the algorithm of \cite{edelsbrunner1986constructing}\;
    Obtain the incidence graph $G(A(H^=))$\;
    Choose a vertex $v_{c'}$ in $G(A(H^=))$ of some arbitrary cell $c'$\;
    Use the greedy algorithm to compute $B_\lambda^*$ at the point $\lambda\in c'$ from the relative interior of $c'$\;
    Use breadth-first search to explore the incidence graph $G(A(H^=))$\;
    \For{each vertex $v_c$ reached with the breadth-first search}{
        Perform an independence test to compute $B_\lambda^*$ for $\lambda\in c$\;
    }
    \Return All cells $c\in A(H^=)$ with corresponding minimum weight basis $B_\lambda^*$\;
    \caption{An improved algorithm for the $p$-parametric matroid problem.}
    \label{alg:p-parametric_matroid_impr_p_1}
\end{algorithm}

\begin{theorem}\label{theo:correct_param_matroid_algo_p_1}
    Let $p\geq2$.
    \cref{alg:p-parametric_matroid_impr_p_1} solves the $p$-parametric matroid problem in time $\O{m^{2p}f(m)}$.
\end{theorem}
\begin{proof}
    We first prove the correctness of the algorithm.
    An edge $(v_{c_1},v_{c_2})$ in $G(A(H^=))$ for two neighboring cells $c_1$ and $c_2$ in $A(H^=)$ belongs to a separating hyperplane $h(e,f)\in H^=$.
    We assume without loss of generality that $c_1\subseteq h^<(e,f)$.
    The sorting of all pairs of weights of elements not equal to $(e,f)$ is the same in $c_1$ as in $c_2$.
    In $c_1$, the weight of $e$ is smaller than that of $f$ and in $c_2$, $f$ has smaller weight than $e$.
    So, if a minimum weight basis $B_\lambda^*$ is known for a point from the relative interior of the cell $c_1$, the minimum weight basis of the cell $c_2$ can be determined by a simple independence test of $B_\lambda^*-e+f$.
    Since the graph $G(A(H^=))$ is connected according to \cref{lem:connectivity_GAH_p_1}, we reach every node $v_c$ and, thus, every cell $c\in A(H^=)$ with the breadth-first search on $G(A(H^=))$.
    
    We now analyze the asymptotic running time.
    As in \cref{alg:p-parametric_matroid_p_1}, the arrangement $A(H^=)$ can be computed in $\O{m^{2p}}$ time.
    The incidence graph $G(A(H^=))$ can be obtained in time linear in the number of faces of dimension $p-1$ of $A(H^=)$ from the incidence lattice $G'(A(H^=))$.
    According to \cref{theo:number_of_faces_p_1}, this number is
    $$\binom{m^2}{1}\cdot\sum_{i=0}^{p-1}\binom{m^2-p-1+p-1}{i}\in\O{m^2 \cdot (m^2)^{p-1}} =\O{m^{2p}}.$$
    The greedy algorithm requires $\O{m\log m+mf(m)}$ time.
    The graph $G(A(H^=))$ has at most $\O{m^{2p}}$ many nodes due to \cref{theo:number_of_cells_p_1}.
    Each edge in $G(A(H^=))$ belongs to a face of dimension $p-1$ in $A(H^=)$, of which there are at most $\O{m^{2p}}$ many.
    Consequently, breadth-first search runs in $\O{m^{2p}}$ time and the computation of the minimum weight basis of a new cell can be performed in $\O{f(m)}$ time with a single independence test.
    Overall, we obtain a total running time of $$\O{m^{2p} + m^{2p}+ m(f(m)+\log m) + m^{2p}f(m)} = \O{m^{2p}f(m)}.$$ 
\end{proof}
The algorithm returns the cells of the arrangement and a unique minimum weight basis for each cell.
If one is interested in a solution set of minimal cardinality, neighboring cells having the same minimum weight basis can be merged, for more information we refer the reader to \cref{alg:weight_set}.

For the multi-parametric minimum spanning tree problem, \cref{alg:p-parametric_matroid_impr_p_1} runs in polynomial time as an independence test, i.~e.\ a cycle test, can be performed in $\O{\log n}$ time, cf.\ \cite{sleator1981data}.
\begin{corollary}
    Let $p\geq2$.
    The $p$-parametric minimum spanning tree problem can be solved in time $\O{m^{2p}\log n}$.
\end{corollary}
We conclude this section with an illustration of the relationship between optimal bases of neighboring cells from \cref{alg:p-parametric_matroid_impr_p_1} for the example of the 2-parametric graphic matroid from \cref{fig:ex_graph_parametric_p_1}.

\begin{figure}
\centering
\begin{tikzpicture}%[scale=0.8]
    \begin{axis}[
        width=1.1\textwidth,
        height=.7\textheight,
        xmin=-3.5,xmax=3.5,
        ymin=-3,ymax=3,
        grid=both,
        axis line style ={line width=.5pt,draw=gray!40},
        major grid style={line width=.2pt,draw=gray!10},
        axis line style={shorten >=-9.5pt, shorten <=-9.5pt},
        axis lines=middle,
        axis line style={latex-latex},
        enlargelimits={abs=0.4},
        ticklabel style={font=\tiny,fill=white}
    ]
    \addplot[fill=black!100,draw opacity=0,fill opacity=0.19] % lo
    coordinates{
        (-3.1,-0.4)
        (-0.15,-0.4)
        (3.1, 2.85)
        (3.1, 3.1)
        (-3.1,3.1)
        };
    \addplot[fill=black!100,draw opacity=0,fill opacity=0.03] %ro
    coordinates{
        (3.1,-0.4)
        (-0.15,-0.4)
        (3.1, 2.85)
        };
    \addplot[fill=black!100,draw opacity=0,fill opacity=0.13] % ru
    coordinates{
        (3.1,-0.4)
        (-0.15,-0.4)
        (-2.85,-3.1)
        (3.1,-3.1)
        };
    \addplot[fill=black!100,draw opacity=0,fill opacity=0.07] %lu
    coordinates{
        (-3.1,-0.4)
        (-3.1, -3.1)
        (-2.85,-3.1)
        (-0.15,-0.4)
        };
    \addplot [line width = 1, smooth, domain=-2.1:3.1] {1-1*x} node[pos=0, above] {$h(e,f)$};
    \addplot [line width = 1, smooth, domain=-2.85:3.1] {-0.25+1*x} node[right] {$h(e,g)$};
    \addplot [line width = 1, smooth, domain=-3.1:3.1] {-0.75+0.25*x} node[pos=0,left] {$h(e,h)$};
    \addplot [line width = 1, smooth, domain=-3.1:3.1] {0.1666+0.3333*x} node[pos=1.07,right, below] {$h(f,g)$};
    \addplot [line width = 1, smooth, domain=-3.1:3.1] {-0.4+0.0*x} node[pos=1.07,right,below] {$h(f,h)$};
    \addplot [line width = 1, smooth, domain=-3.1:3.1] {-1.25-0.5*x} node[pos=0,left] {$h(g,h)$};

    \node[circle,fill,inner sep=1.4] at (-0.6,-0.6) {};
    \node at (-0.85,-0.65) {\footnotesize$T^*$};
    \node at (-1.8,-0.7) {\footnotesize$T^*$};
    \node at (-1.9,-1.7) {\footnotesize$T^*$};
    \node at (-2.5,-0.55) {\footnotesize$T^*$};
    \node at (0.3,-0.5) {\footnotesize$T^*-e+g$};
    \node at (-0.6,-1.6) {\footnotesize$T^*-e+g$};
    \node at (1,-1.2) {\footnotesize$T^*-e+g$};
    \node at (2.45,-0.75) {\footnotesize$T^*-e+g$};
    \node at (0.65,-0.06) {\footnotesize$T^*$};
    \node at (0.61,-0.31) {\footnotesize$-h+f-e+g$};
    \node at (1.8,0.3) {\footnotesize$T^*-h+f-e+g$};
    \node at (2.3,1.5) {\footnotesize$T^*$};
    \node at (2.25,1.25) {\footnotesize$-h+f-e+g$};
    \node at (2.77,-0.17) {\footnotesize$T^*$};
    \node at (2.45,-0.32) {\footnotesize$-h+f-e+g$};
    \node at (0.6,1.7) {\footnotesize$T^*-h+f$};
    \node at (-1.8,1.2) {\footnotesize$T^*-h+f$};
    \node at (-2.6,-0.1) {\footnotesize$T^*$};
    \node at (-2.6,-0.3) {\footnotesize$-h+f$};
    \node at (-0.6,-0.3) {\footnotesize$T^*-h+f$};

    \node at (3.7,0.15) {$\lambda_1$};
    \node at (-0.2,3.31) {$\lambda_2$};
    
    \end{axis}
\end{tikzpicture}
\caption{The arrangement of the separating hyperplanes $A(H^=)$ obtained by intersecting the 2-parametric weights of the edges from the graph in \cref{fig:ex_graph_parametric_p_1}.
Each cell is assigned its unique minimum spanning tree.
The minimum spanning trees of two neighboring cells differ by a maximum of two edges.
If two cells have the same minimum spanning tree, they are colored in the same shade of gray.}
\label{fig:neighborhood_p_1}
\end{figure}
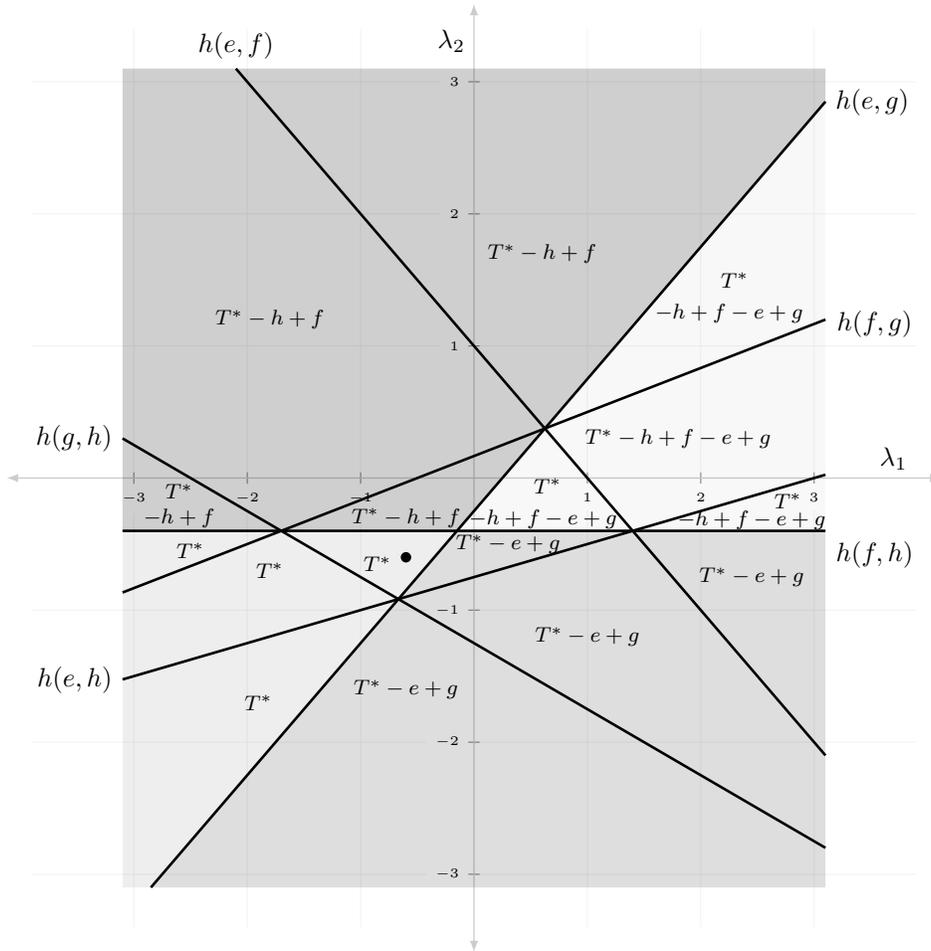

\begin{example}\label{ex:neighborhood_bases_p_1}
    \cref{fig:neighborhood_p_1} shows the arrangement of hyperplanes $A(H^=)$, which is obtained by intersecting the weight functions of the edges of the graph from \cref{fig:ex_graph_parametric_p_1}.
    With \cref{alg:p-parametric_matroid_impr_p_1}, we can compute a solution of the 2-parametric minimum spanning tree problem and the corresponding optimal value function $w$, see \cref{fig:opt_value_fct_p_1}.
    First, the minimum spanning tree $T^*=\Set{e,h}$ is computed for the cell containing the black point.
    Starting from this cell, the breadth-first search is used to explore the incidence graph of the arrangement.
    For each new cell, the minimum spanning tree can be computed by a cycle test of the minimum spanning tree of the neighboring cell with a swap of the edges that form the corresponding separating hyperplane.
    In \cref{fig:neighborhood_p_1}, the minimum spanning trees of two neighboring cells have a symmetric difference of at most two, and if the symmetric difference is zero, the cells are subsumed in the same gray scale.
    The four resulting merged cells define the 2-dimensional polyhedra in the graph of the optimal value function $w$ from \cref{fig:opt_value_fct_p_1}.
    More precisely, the projection of the graph of $w$ onto the $\lambda_1$-$\lambda_2$-plane equals the subdivision of the parameter interval into these four merged cells.    
\end{example}

\section{Application 1: Weight set decomposition for multi-objective matroids}\label{sec:wsd_p_1}
In this section, we apply our \cref{alg:p-parametric_matroid_impr_p_1} to the multi-objective version of the matroid problem (\cite{ehrgott1996matroids}) and its special variant - the well-known minium spanning tree problem, \cite{hamacher1994spanning,ruzika2009survey}.
For the latter, we obtain a polynomial time algorithm for the computation of all extreme supported spanning trees and the corresponding weight set decomposition.
The breadth-first search technique on the cells of the arrangement of the separating hyperplanes in combination with the adjacency of optimal bases improves upon the algorithm of \cite{seipp2013adjacency} by a factor of $\O{m\frac{\alpha(m)}{\log n}}$.
Here $\alpha$ is a functional inverse of Ackermann's function.
Our algorithm is faster than other weight set decomposition methods that allow for an easy and explicit computation of their running time.

The concept of adjacency has been used before in the multi-objective (graphic) matroid problem.
For arbitrary matroids, \cite{ehrgott1996matroids} has established a characterization of the efficient bases via a basis exchange property.
Also for graphic matroids, adjacency is a widely used concept in the multicriteria minimum spanning tree problem, cf.\ \cite{hamacher1994spanning,ehrgott1997connectedness,ruzika2009survey,gorski2011connectedness}.
For the bi-objective case, \cite{aggarwal1982minimal} and \cite{andersen1996bicriterion} used adjacent trees for the exact or approximate computation of the efficient spanning trees, respectively.
Also \cite{correia2021finding} used adjacency for the weight set decomposition in the general multi-objective minimum spanning tree problem.

We start this section with an introduction of the required notation from multi-objective optimization, cf.\ \cite{ehrgott2005multicriteria}.
For two vectors $y,y'\in\R^p$, we define the following ordering relations
\begin{enumerate}
    \item[] $y\leqq y'$ if and only if $y_i \leq y_i'$ for all $i=1,\dotsc,p$,
    \item[] $y\leq y'$ if and only if $y\leqq y'$ and $y\neq y'$,
    \item[] $y< y'$ if and only if $y_i < y_i'$ for all $i=1,\dotsc,p$.
\end{enumerate}
The set of all non-negative vectors in $\R^p$ is defined as $\R_\geqq^p\coloneqq\Set{\lambda\in\R^p\colon\lambda\geqq0}$ and $\R_>^p\coloneqq \Set{\lambda\in\R^p\colon\lambda>0}$ summarizes all strictly positive vectors in $\R^p$.

Each edge $e\in E$ is associated with a \emph{cost vector} $c_e\coloneqq(c_{e,1},\dotsc,c_{e,p})\in \R^p$, where $p\geq2$ is again a fixed natural number.
The \emph{cost of a tree} $T\in\T$ is then defined as
$$c(T)\coloneqq(c_1(T),\dotsc,c_p(T))^\top\coloneqq\big(\sum_{e\in T}c_{e,1},\dotsc,\sum_{e\in B}c_{e,p}\big)^\top$$
and the multi-objective minimum spanning tree problem reads as follows.
\begin{problem}[Multi-objective minimum spanning tree problem]\label{prob:multiobjective_mst}
    The problem
    \begin{mini*}[2]
	{}{c(T)}{}{}
	\addConstraint{T\in\mathcal{T}}{}{}
    \end{mini*}
    is called multi-objective minimum spanning tree problem, where $\min$ means that each of the $p$ objectives is to be minimized.
\end{problem}
Each spanning tree $T\in\T$ is mapped to its image $c(T)\in\R^p$ and these images define the image set $Y\coloneqq\Set{c(T)\colon T\in\mathcal{T}}$.
Since there is generally no spanning tree that minimizes all objectives simultaneously, we are interested in so-called efficient trees.
A spanning tree $T\in\T$ is \emph{efficient} if there is no other spanning tree
$T'\in\T$ with $c(T')\leq c(T)$.
In this case, the corresponding image $c(T)$ is called \emph{non-dominated point}.
We now establish the relationship to the multi-parametric minimum spanning tree problem.
\begin{definition}[Weighted sum scalarization (\cite{ehrgott2005multicriteria})]\label{prob:weighted_sum_matroid}
    Let $\lambda\in\R_\geqq^p$.
    The single-objective problem
    \begin{mini*}[2]
	{}{\lambda^\top c(T)=\sum_{i=1}^p\lambda_i c_i(T)}{}{}
	\addConstraint{T\in\T}{}{}
    \end{mini*}
    is called the \emph{weighted sum problem} with respect to $\lambda$.
\end{definition}
An efficient tree which can be computed with the weighted sum problem for some $\lambda\in\R_\geqq^p$ is called \emph{supported efficient}.
The corresponding image is called supported non-dominated point.
A supported non-dominated point which is also an extreme point of the convex hull of the image set $Y$ is called \emph{extreme supported non-dominated}.
The set of all extreme supported non-dominated points is denoted by $Y_{exN}$.
It is well known that there is a one-to-one correspondence between extreme supported non-dominated points and the optimal solutions of the weighted sum problem with strictly positive weights.
\begin{theorem}[\cite{seipp2013adjacency}]\label{theo:extreme_points_ws_p_1}
    The image $y=c(T)$ of a tree $T\in\T$ is an extreme supported non-dominated point if and only if there exists a strictly positive weighting vector $\lambda\in\R_>^p$ such that $y$ is the unique optimal solution of the weighted sum problem $\min_{y'\in Y} \lambda^\top y'$.
\end{theorem}
A solution of the multi-parametric minimum spanning tree problem with parameter interval $I=\R_>^p$ thus provides all extreme supported non-dominated points.
Positive scalar multiples of weight vectors lead to same solutions of the weighted sum problem and one is therefore interested in the set of normalized weight vectors.
\begin{definition}[Weight set (\cite{przybylski2010recursive})]
    The set
    $$\Lambda\coloneqq\Set{\lambda\in\R_>^p\colon \sum_{i=1}^p\lambda_i=1}$$
    of all normalized strictly positive weight vectors is called \emph{weight set}.
\end{definition}
The dimension of $\Lambda$ is $p-1$ as the last component $\lambda_p=1-\sum_{i=1}^{p-1}\lambda_i$ reduces the degree of freedom by one.
This yields a projection of the weight space in $\R^p$ onto $\R^{p-1}$.
Even if positive scalar multiples of weights are no longer taken into account, different weights may still lead to the same spanning tree.
One is therefore interested in the so-called \emph{weight set decomposition}, which divides $\Lambda$ into polytopes with the following property.
All inner points of a polytope lead to the same unique solution of a weighted sum problem, where the associated objective function vector $y\in Y$ summarizes all these weights in its weight set component $\Lambda(y)$.
It is known from the literature that the extreme supported non-dominated points are sufficient to determine all weight set components of $\Lambda$, see \cite{benson2000outcome}.
Thus, the \emph{weight set component} of a point $y\in Y$ can be written as 
$$\Lambda(y)=\Set{\lambda\in \Lambda\colon \lambda^\top y\leq\lambda^\top y'\; \text{for all}\;y'\in Y_{exN}}\; \text{and}\;\; \Lambda=\bigcup_{y\in Y_{exN}}\Lambda(y)$$
then yields the weight set decomposition.
We refer the reader to \cite{seipp2013adjacency} for more details on the weight set decomposition of the multi-objective minimum spanning tree problem.
The author has developed a polynomial time algorithm to compute the weight set decomposition, which we extend by the neighborhood search from \cref{alg:p-parametric_matroid_impr_p_1} in the following.
To do this, we need to investigate a few aspects in more detail.

As shown in \cite{seipp2013adjacency}, it is sufficient to consider the cells of the arrangement $A(H^=)$ that have a non-empty intersection with the affine hull of the weight set $\Lambda$.
These are obtained by intersecting all separating hyperplanes from the subset  
$$\Bar{H}\coloneqq\Set{h(e,f)\in H^=\colon c_e\nleqq c_f,\; c_e\ngeqq c_f}.$$
Intuitively, if $c_e\leqq c_f$, then the edge $e$ is not more expensive than $f$ for each relevant vector $\lambda\in \Lambda$, which means that the separating hyperplane $h(e,f)$ does not have to be taken into account.
The projection of the intersection of the hyperplanes $h(e,f)\in \Bar{H}$ with the affine hull of $\Lambda$ onto the first $p-1$ components thus yields an arrangement of at most $\binom{m}{2}$ hyperplanes in $\R^{p-1}$.
In contrast to the parametric (graphic) matroid problem, where the case $\I=\R^p$ is possible, the number of cells and, thus, the running time of the algorithm is reduced by a factor of $m^2$.

In the algorithm of \cite{edelsbrunner1986constructing}, it may happen that the representative inner point of a cell that lies partially outside the weight set $\Lambda$ is not contained in $\Lambda$.
To ensure that a representative is always found within the set $\Lambda$, the $p$ hyperplanes $b_i\coloneqq\Set{\lambda\in\R^p\colon\lambda_i=0}$ for $i=1,\dotsc,p$, which define the boundary of $\Lambda$, are added to the arrangement.
This does not increase the asymptotic number of hyperplanes and, thus, of cells of the arrangement.

Note, that the corresponding incidence graph $G(A(\Bar{H}))$ is closely related to the dual graph of the weight set decomposition, see \cite{zimmer1997fast,przybylski2010recursive}.
The only difference is that the dual graph subsumes polytopes with the same unique minimum spanning tree.
After executing our algorithm for the computation of the weight set decomposition, these polytopes are merged and, then, the two graphs coincide.
Again, the incidence graph $G(A(\Bar{H}))$ can be computed from the incidence lattice $G'(A(\Bar{H}))$ from the algorithm of \cite{edelsbrunner1986constructing} in time linear in the number of cells of $A(\Bar{H})$.
Our procedure for the computation of the weight set decomposition is summarized in \cref{alg:weight_set}.

\begin{algorithm}[!tp]
    \LinesNumbered
    \Input{A graph $G=(V,E)$ with edge costs $c_e\in\R^p$.}
    \Output{The set $Y_{exN}$ and the weight set decomposition $W^0=\bigcup_{y\in Y_{exN}}W^0(y)$.}
    
    Compute the relevant subset $\Bar{H}=\Set{h(e,f)\in H^=\colon c_e\nleqq c_f\; c_e\ngeqq c_f}$\;
    Add the hyperplanes $b_i=\Set{\lambda\in\R^p\colon\lambda_i=0}$ for $i=1,\dotsc,p$\; 
    Compute $A(\Bar{H})$ with the algorithm of \cite{edelsbrunner1986constructing}\;
    Obtain the incidence graph $G(A(\Bar{H}))$\;
    Choose a vertex $v_{c'}$ in $G(A(\Bar{H}))$ of some arbitrary cell $c'$\;
    Compute a minimum spanning tree $T^*$ at the point $\lambda\in c'$ from the relative interior of $c'$\;
    Initialize $Y_{exN}\algass\Set{c(T^*)}$\;
    Use breadth-first search to explore $G(A(\Bar{H}))$\;
    \For{each vertex $v_c$ reached with the breadth-first search}{
        Perform a cycle test to compute a minimum spanning tree $T^*$ in $c$\;
        \If{$c(T^*)\notin Y_{exN}$}{
            Update $Y_{exN}\algass Y_{exN}\cup\Set{c(T^*)}$\;
            Store $(W^0(c(T^*)),\lambda_{c(T)})\coloneqq(R,\lambda)$
        }
        \Else{
            Set $z\algass c(T^*)\in Y_{exN}$\;
            Set $(W^0(z),\lambda_z) \algass (W^0(z)\cup R,\lambda_z)$
        }
    }
    \Return $Y_{exN}$ and $W^0(y)$ for all $y\in Y_{exN}$\;
    \caption{An algorithm for the weight set decomposition.}
    \label{alg:weight_set}
\end{algorithm}

\begin{theorem}
    \cref{alg:weight_set} computes the set $Y_{exN}$ of all extreme supported non-dominated points and the corresponding weight set decomposition for the multi-objective minimum spanning tree problem in time $\O{m^{2p-2}\log n}$.
\end{theorem}
\begin{proof}
    If we intersect the cells of the arrangement $A(\Bar{H})$ with the weight space $\Lambda$ and project onto the first $p-1$ components, we obtain the polytopes of the weight set decomposition.
    According to \cref{theo:extreme_points_ws_p_1}, each such polytope corresponds to an extreme supported non-dominated image of a spanning tree, which is the optimal solution of the weighted sum problem for all $\lambda$ of the polytope.
    If two neighboring polytopes lead to the same minimum spanning tree, they are merged in line 16 of the algorithm.
    The connectivity of the incidence graph $G(A(\Bar{H}))$ guarantees that the breadth-first search finds all polytopes and thus all extreme supported non-dominated points $y\in Y_{exN}$.
    Correctness thus follows analogously to \cref{theo:correct_param_matroid_algo_p_1}.

    Since the set $\Bar{H}$ provides an arrangement of at most $\binom{m}{2}$ hyperplanes in $\R^{p-1}$, we obtain a running time of $\O{m^{2(p-1)}}$ for the first seven lines in analogy to \cref{theo:correct_param_matroid_algo_p_1}.
    The addition of the $p$ hyperplanes $b_i$ does not increase the asymptotic number of hyperplanes and thus the asymptotic running time.
    The incidence graph $G(A(\Bar{H}))$ has at most $\O{m^{2(p-1)}}$ many nodes and
    $$\binom{m^2}{1}\sum_{i=0}^{p-2}\binom{m^2-(p-1)-1+p-2}{i}\in\O{m^2 \cdot (m^2)^{p-2}} =\O{m^{2(p-1)}}$$ many edges according to \cref{theo:number_of_faces_p_1}.
    Thus, the running time of the breadth-first search and the number of iterations of the for-loop are bounded by $\O{m^{2(p-1)}}$.
    A cycle test can be performed in $\O{\log n}$ time with the dynamic tree data structure of \cite{sleator1981data}.
    We store each polytope of the weight set decomposition via its edges in the incidence graph in a linked list.
    If two neighboring polytopes are merged, the edge corresponding to the common face of dimension $p-2$ is deleted and the two linked lists are merged, which can be done in $\O{1}$.
    In summary, we obtain a total running time of $$\O{m^{2(p-1)}} + \O{m^{2(p-1)}\log n}= \O{m^{2p-2}\log n}.$$
\end{proof}
The algorithm of \cite{seipp2013adjacency} uses Kruskal's algorithm with $\O{m\alpha(m)}$ running time (\cite{chazelle2000minimum}) for the computation of the minimum spanning tree of each polytope in the weight set decomposition.
Hence, his algorithm has a total running time of $\O{m^{2p-1}\alpha(m)}$, which we improve by a factor of $\O{m\frac{\alpha(m)}{\log n}}$.

We note that also \cite{correia2021finding} have used adjacency of trees for the computation of the weight set decomposition.
However, their algorithm lists all supported efficient spanning trees, which are in general exponentially many (\cite{ruzika2009survey}) and, consequently, their algorithm has no polynomial running time.
There are other, more general algorithms for the weight set decomposition that do not specifically exploit the matroid structure, cf.\ \cite{benson2000outcome, przybylski2010recursive, halffmann2020inner}.
For these methods, the explicit computation of the asymptotic running time is often a complex task, especially for $p>3$ objectives.
To summarize, our algorithm asymptotically outperforms the weight set decomposition methods that allow an explicit computation of their running time.
\begin{remark}[Multi-objective matroid problem]
    We note that our algorithm can also be used for the computation of the weight set decomposition in the general multi-objective matroid problem, cf.\ \cite{ehrgott1996matroids}.
    
    Each element $e\in E$ of the matroid's ground set is associated with a cost vector $c_e\coloneqq(c_{e,1},\dotsc,c_{e,p})\in \R^p$ and the cost of a basis $B$ of $\M$ is defined as
    $$c(B)\coloneqq\big(\sum_{e\in B}c_{e,1},\dotsc,\sum_{e\in B}c_{e,p}\big).$$
    Similar to \cref{prob:multiobjective_mst}, the goal is to minimize the vector-valued cost $c(B)$ over the set of all bases $B$ of $\M$.
    The notion of (extreme supported) non-dominated points and the concept of weight set decomposition can be formulated analogously as in the case of graphic matroids.
    The adaptation of \cref{alg:weight_set} to arbitrary matroids provides all extreme supported non-dominated points and the weight set decomposition in $\O{m^{2p-2}f(m)}$ time.
\end{remark}

\section{Application 2: Multi-parametric matroid one-interdiction}\label{sec:interdiciton_p_1}
In this section, we apply our algorithm for the $p$-parametric matroid problem to solve the corresponding one-interdiction version.
In this so-called $p$-parametric matroid one-interdiction problem, 
each element of the matroid's ground set is again associated with a weight that depends linearly on $p$ parameters from a given parameter interval $\I\subseteq\R^p$.
Here, we aim at computing, for each parameter vector $\lambda\in\I$, a so-called most vital element and the corresponding objective function value, i.~e.\ the weight of an optimal one-interdicted minimum weight basis.
A most vital element at $\lambda\in\I$ is an element that maximally increases the weight of a minimum weight basis when removed from the matroid's ground set.
A solution to the problem is given by a decomposition of the parameter interval $\I$ into subsets, such that each subset has a unique element that is the most vital element for all parameter vectors within the subset.

In the literature, two parametric matroid interdiction problems have recently been introduced. 
In \cite{hausbrandt2024parametric}, the authors considered the one-parametric version of the problem and in \cite{hausbrandt2024theparametric}, we enhanced this theory by allowing an arbitrary but fixed number of elements to be interdicted.
We refer the reader to these two articles for a more detailed list of relevant references.

This section can be seen as an extension of the work in \cite{hausbrandt2024parametric} with the main difference that the weights of the elements do not only depend on one but on an arbitrary but fixed number $p$ of parameters.
We show that the solution of this problem is of polynomial size, i.~e.\ it consists of polynomially many subsets decomposing $\I$ with associated unique most vital element.
Further, we show that the solution can be computed in polynomial time using \cref{alg:p-parametric_matroid_impr_p_1} whenever an independence test can be performed in time polynomial in the input size.

In this section, we additionally introduce the following notation.
We write $\M_e$ for the matroid $(E-e,\Set{F\in\F\colon F\subseteq E-e})$ to indicate that element $e$ is removed.
The next definitions extend the definitions in \cite{hausbrandt2024parametric} to the multi-parametric case.
\begin{definition}[Most vital element] \label{def:mve_p_1}
    For an element $e \in E$, we denote by $B_\lambda^e$ a minimum weight basis of $\M_e$ at $\lambda\in I$.
    If $\M_e$ does not have a basis of rank $k$, we set $w(B_\lambda^e,\lambda)=\infty$ for all $\lambda\in I$.
    An element $e^*\in E$ is called \emph{most vital element} at $\lambda$ if $w(B_\lambda^{e^*},\lambda) \geq w(B_\lambda^e,\lambda)$ for all $e\in E$.
\end{definition}
In addition to the computation of a most vital element for each parameter vector $\lambda\in\I$, we are also
interested in the corresponding objective function value.
\begin{definition}[Optimal interdiction value function] \label{def:opt_ int_value_fct_p_1}
    For an element $e\in E$, we define the function $y_e$ via $y_e\colon I\to\mathbb{R}$, $\lambda\mapsto w(B_\lambda^e,\lambda)$ mapping the parameter vector $\lambda$ to the weight of a minimum weight basis of $\M_e$ at $\lambda$. 
    We denote the weight of an optimal one-interdicted minimum weight basis at $\lambda$ by $y(\lambda)\coloneqq\max\{y_e(\lambda)\colon\, e\in E\}$.
    The \emph{optimal interdiction value function} $y$ is then defined as $y\colon I\to\mathbb{R}$, $\lambda\mapsto y(\lambda)$.
\end{definition}
We are ready to formally state the $p$-parametric matroid one-interdiction problem.
\begin{problem}[$p$-Parametric matroid one-interdiction problem]\label{prob:p_1}
    Given a matroid~$\M$ with $p$-parametric weights $w(e,\lambda)$ and the parameter interval~$I\subseteq\R^p$, compute, for every $\lambda\in I$, a most vital element $e^*$ and the corresponding objective function value $y(\lambda) = y_{e^*}(\lambda)$.
\end{problem}
In \cref{prob:p_1}, the goal is to compute the upper envelope $y$, i.~e.\ the point-wise maximum, of all functions $y_e$ and, for each parameter vector $\lambda\in\I$, a maximizer $y_{e^*}$ with corresponding element $e^*$.
To exclude trivial cases, we make the following assumption in this section in addition to \cref{ass:p_1}.
\begin{assumption}\label{ass:additionally_interdiction_p_1}
    There exists a basis $B_\lambda^e$ with rank $k$ for every $e\in E$ and $\lambda\in I$.
\end{assumption}
If there is no basis with rank $k$ after interdicting an element $e\in E$, \cref{prob:p_1} can be solved trivially.
Therefore, \cref{ass:additionally_interdiction_p_1} can be made without loss of generality.
Consequently, each function $y_e$ is a continuous, piecewise-linear and concave optimal value function on $M_e$, with continuity being transferred to the optimal interdiction value function $y$.
Note, however, that in contrast to the graphs of the functions $y_e$, the graph of $y$ does not generally consist of $p$-dimensional polyhedra.
Instead, the graph of $y$ consists of a collection of $p$-dimensional surface patches, which are generally not convex.
It is crucial to bound the number of these $p$-dimensional surface patches in the graph of $y$, since the most vital element changes with each such patch.
The number thus serves as a natural complexity measure for the $p$-parametric matroid one-interdiction problem.

To obtain a polynomial bound on the number of $p$-dimensional surface patches in the graph of $y$, we use a theorem from \cite{sharir1994almost} that limits the complexity of the lower envelope of general surface patches.
Roughly spoken, the complexity of the lower envelope of surface patches is given by its number of faces of all dimensions.
By symmetry, we can apply the theorem of \cite{sharir1994almost} to bound the   number of $p$-dimensional surface patches in the upper envelope $y$.

Under an assumption that we consider in more detail in the subsequent proof, the theorem states the following.
The complexity of the upper envelope of $q$ surface patches in $\R^d$ is in $\O{q^{d-1+\varepsilon}}$, for any $\varepsilon>0$, i.~e.\ it is arbitrarily close to $\O{q^{d-1}}$.
\begin{theorem}\label{theo:bound_patches_of_y_p_1}
    The number of $p$-dimensional surface patches in the graph of the optimal interdiction value function $y$ is bounded by $\O{m^{(2p+1)(p+\varepsilon)}}$, for any $\varepsilon>0$.
\end{theorem}
\begin{proof}
    We require the number of $p$-dimensional surface patches in the upper envelope $y$ of the graphs of the functions $y_e$ for $e\in E$.
    The graph of each of these $m$ many functions contains at most $\O{m^{2p}}$ many $p$-dimensional polyhedra due to \cref{lem:opt_value_fct_concave_p_1}.
    The theorem of \cite{sharir1994almost} implies that the required number of surface patches is in $\O{m^{(2p+1)(p+\varepsilon)}}$, for any $\varepsilon>0$, provided we have the following assumption.
    Each $p$-dimensional polyhedron $P$ in the graph of a function $y_e$ is of the form $\bigl(Q=0\bigr) \wedge \bigl(B\left(R_1\sigma_1 0,\dotsc,R_u\sigma_u 0\right)\bigr),$ where $u$ is a constant, $\sigma_1,\dotsc,\sigma_u\in\Set{\geq,\leq}$, $B$ is a Boolean formula, $Q$ and $R_1,\dotsc,R_u \in \R[x_1,\dotsc,x_{p+1}]$ are polynomials and the degrees of $Q$ and $R_1,\dotsc,R_u$ are of constant size.
    
    For a given instance of \cref{prob:p_1} and a given polyhedron $P$ in the graph of a function $y_e$, there is a fixed number of hyperplanes $h_1,\dotsc,h_u\in H^=$ defining $P$.
    Further, if $h(e,f)\in H^=$ is a separating hyperplane constituting a $(p-1)$-dimensional face of $P$, then one of the half-spaces $h^\geq(e,f)\coloneqq\Set{\lambda\in\R^p\colon w(e,\lambda)\geq w(f,\lambda)}$ or $h^\leq(e,f)\coloneqq\Set{\lambda\in\R^p\colon w(e,\lambda)\leq w(f,\lambda)}$ implies a polynomial of the required form.
    Without loss of generality, let $P$ be contained in the half-space $h^\geq(e,f)$.
    From $w(e,\lambda)\geq w(f,\lambda)$, we obtain the inequality $$R\coloneqq a_e-a_f+\lambda_1(b_{1,e}-b_{1,f})+\dotsc+\lambda_p (b_{p,e}-b_{p,f})\geq 0,$$ where $R$ is a polynomial of degree one of the desired form.
    To summarize, $P$ can be represented in the required form and the theorem of \cite{sharir1994almost} proves the bound.
\end{proof}
We now develop an algorithm that solves \cref{prob:p_1} in polynomial time whenever an independence test can be performed in time polynomial in the input size.
We first use \cref{alg:p-parametric_matroid_impr_p_1} on the matroid $\M_e$ for each $e\in E$ to obtain the functions $y_e$.
Since a most vital element $e^*$ at $\lambda\in\I$ is an element for which $y_{e^*}(\lambda)=y(\lambda)$, it suffices to compute the upper envelope $y$ of the functions $y_e$.
Unfortunately, in dimension $d\geq4$, there exist only algorithms for the computation of the upper envelope of surface patches if these are simplices, i.~e.\ if they have exactly $d+1$ vertices (\cite{agarwal2000arrangements}).
Consequently, we cannot compute the upper envelope $y$ directly.
Instead, we have to triangulate each polyhedron in the graph of all functions $y_e$, meaning we have to decompose them into simplices.
For more details on triangulations, we refer the reader to the textbook of \cite{deloera2010triangulations}.
Afterwards, we use the algorithm of \cite{edelsbrunner1989upper} to compute the upper envelope of all simplices obtained.
This provides us with a representation of $y$, which may consist of more $p$-dimensional surface patches than necessary due to the triangulations of the polyhedra in the graphs of the functions $y_e$.
Clearly, this may increase the asymptotic running time of our algorithm.
However, we show in the next theorem that polynomial running time is retained.
Note, that we store the corresponding element $e$ for each simplex of the graph of a function $y_e$.
This allows us to obtain a most vital element for each $\lambda\in\I$ during the computation of the upper envelope $y$.

The procedure is summarized in \cref{alg:p-parametric_interdiction_p_1}.
\cref{theo:algo_interdiction_p_1} proves correctness and analyzes its running time.

\begin{algorithm}[t]
    \LinesNumbered
    \Input{A matroid \M{} with $p$-parametric weights $w(e,\lambda)$ and parameter interval~$I \subseteq \mathbb{R}^p$.}
    \Output{A representation of the upper envelope of $y$.}
    
    \For{each $e\in E$}{
        Use \cref{alg:p-parametric_matroid_impr_p_1} on $\M_e$ to obtain the function $y_e$\;
    }
    Triangulate each polyhedron in the graph of each function $y_e$\;
    Store the corresponding element $e$ for each simplex obtained\;
    Compute the upper envelope $y$ of all simplices of all graphs of the functions $y_e$\;
    \Return $y$\;
    \caption{A polynomial time algorithm for the $p$-parametric matroid one-interdiction problem.}
    \label{alg:p-parametric_interdiction_p_1}
\end{algorithm}

\begin{theorem}\label{theo:algo_interdiction_p_1}
    \cref{alg:p-parametric_interdiction_p_1} correctly solves the $p$-parametric matroid one-interdiction problem in time $\mathcal{O}\bigl( m^{2p+1} \bigl( m^{p^3+3p^2-p-1} + f(m) \bigr)\bigr)$.
\end{theorem}
\begin{proof}
    For each $e\in E$, we compute the function $y_e$ on $M_e$ with \cref{alg:p-parametric_matroid_impr_p_1} in $\O{m^{2p}f(m)}$ time.
    In order to compute their upper envelope $y$, we first have to triangulate each polyhedron in the graph of each function $y_e$.
    It follows from \cite[Theorem~2.6.1]{deloera2010triangulations} that the number of simplices needed to triangulate one of the $p$-dimensional polyhedra is bounded by $\O{s^{\lfloor\frac{p+2}{2}\rfloor}}$, where $s$ is the number of vertices of the corresponding polyhedron.
    This number equals the number of 0-faces of the corresponding cell in the arrangement $A(H^=)$.
    \cref{theo:number_of_faces_p_1} implies that $s$ is bounded by $\binom{m^2}{p} \sum_{i=0}^0\binom{m^2-p-1+0}{i} \in \O{m^{2p}}$ yielding $\O{m^{p(p+2)}}$ many simplices per polyhedron.
    Then, by Lemma~8.2.2 from \cite{deloera2010triangulations} it follows that one of the polyhedra can be triangulated in time
    $$\O{(p+2)\cdot (m^{2p})^2 \cdot m^{p(p+2)}}=\O{p\cdot m^{p(p+6)}}.$$
    According to \cref{cor:opt_value_fct_p_1}, there are a total of $\O{m\cdot m^{2p}}$ many $p$-dimensional polyhedra, which leads to a total running time of $\O{m^{2p+1}\cdot p\cdot m^{p(p+6)}}=\O{p\cdot m^{p^2+8p+1}}$ for the triangulations.

    The number of simplices in all graphs of all functions $y_e$ is bounded by $\O{m^{p^2+3p+1}}$.
    The upper envelope of $r$ many $p$-dimensional simplices in $\R^{p+1}$ can be computed with the algorithm of \cite{edelsbrunner1989upper} in $\O{r^p\alpha(r)}$ time, cf.\ \cite{agarwal2000arrangements}.
    We obtain a running time of
    $$\O{(m^{p^2+3p+1})^p\cdot \alpha(m^{p^2+3p+1})}=\O{m^{p(p^2+3p+1)}\alpha(m)}$$ for the computation of the upper envelope $y$.
    Once we have a representation of $y$, we directly obtain, for each $p$-dimensional surface patch in the graph of $y$, an element that is a most vital element for all values within the corresponding surface patch.
    In summary, \cref{alg:p-parametric_interdiction_p_1} runs in time
    \begin{align*}
        &\mathcal{O}\bigl(m\cdot m^{2p}f(m) + p\cdot m^{p^2+8p+1} + m^{p(p^2+3p+1)}\alpha(m)\bigr) \\
        =\; &\mathcal{O}\bigl( m^{2p+1} \bigl( m^{p^3+3p^2-p-1} + f(m) \bigr)\bigr).
    \end{align*}
\end{proof}
To summarize, a solution of the $p$-parametric matroid one-interdiction problem consists of a subdivision of the parameter interval $\I\subseteq\R^p$ into polynomially many subsets with one corresponding most vital element per subset.
Furthermore, the solution can be computed with a polynomial number of independence tests.

\section{Conclusion}
This article proposes an algorithm that solves the multi-parametric matroid problem with a polynomial number of independence tests.
We provide the first method that works for arbitrary matroids, where each element has a weight that depends linearly on an arbitrary but fixed number of parameters coming from real intervals.
We show that a solution to the problem consists of a polynomial number of polyhedra decomposing the multi-dimensional parameter interval such that each polyhedron is associated with a unique basis that is optimal for all parameter vectors within the corresponding polyhedron.
Applied in a multi-objective context, our algorithm computes the set of all extreme supported bases as well as the weight set decomposition for the multi-objective matroid problem.
For the special case of minimum spanning trees, we obtain an asymptotically faster algorithm than the methods that allow for an explicit computation of their running time, cf.\ \cite{seipp2013adjacency, correia2021finding}.
Finally, we showed that our algorithm can be used to solve the multi-parametric matroid one-interdiction problem.
Instead of a minimum weight basis, the goal is to compute, for each parameter vector, one element that maximizes its weight when removed from the matroid.
Whenever an independence test can be performed in time polynomial in the input size of the problem, we obtain polynomial running time for this special matroid interdiction variant.

\section*{Acknowledgments}
This work was partially supported by the project "Ageing Smart -- Räume intelligent gestalten" funded by the Carl Zeiss Foundation and the DFG grant RU 1524/8-1, project number 508981269.
The authors would like to thank Levin Nemesch for helpful discussions and the resulting improvement of the article.

\bibliography{bib_matroid}
	
\end{document}